\documentclass[smallextended,referee,envcountsect]{svjour3} 
 
\smartqed 

\usepackage{graphicx}

\journalname{}

\usepackage{amsmath}
\usepackage{amsfonts}
\usepackage{mathrsfs}
\usepackage{enumitem}
\usepackage{scrextend}

\newcommand{\marginpersonal}[3]
{\begin{addmargin}[#1em]{#2em}
		#3
\end{addmargin}}

\newcommand{\margin}[1]
{\begin{addmargin}[2em]{1em}
		#1
\end{addmargin}}

\newcommand{\enurom}[1]
{\vspace{-2mm}\begin{enumerate}[label=\rm{(\roman*)}]
		#1
\end{enumerate}}

\newcommand{\enualp}[1]
{\vspace{-2mm}\begin{enumerate}[label=\rm{(\alph*)}]
		#1
\end{enumerate}}

\newcommand{\ccd}[0]
{(.)}

\newcommand{\ttnn}[1]
{\textnormal{#1}}

\newcommand{\modulo}[1]
{\left|#1\right|}

\newcommand{\rif}[1]
{(\ref{#1})}

\newcommand{\norm}[1]
{\left\| #1 \right\|}

\newcommand{\ra}
{\rightarrow}

\newcommand{\ps}[2]
{\langle\,#1,#2\rangle}

\newcommand{\rt}[0]
{(t)}

\newcommand{\rs}[0]
{(s)}

\newcommand{\mineq}[1]
{\leq #1}

\newcommand{\mageq}[1]
{\geq #1}

\newcommand{\equazioneref}[2]
{
		\begin{eqnarray}\label{#1}
		\begin{split}
		#2
		\end{split}
		\end{eqnarray}
}

\newcommand{\sistemaref}[2]
{
		\begin{eqnarray}
		\begin{cases}\label{#1}
		#2
		\end{cases}
		\end{eqnarray}
}

\newcommand{\sistemanoref}[1]
{
		\begin{eqnarray*}
			\begin{cases}
				#1
			\end{cases}
		\end{eqnarray*}
}

\newcommand{\eee}[1]
{
		\begin{eqnarray*}\begin{split}
				#1
		\end{split}\end{eqnarray*}
}

\newcommand{\graffe}[1]
{\left\{ #1 \right\}}

\newcommand{\tonde}[1]
{\left ( #1 \right )}

\newcommand{\eps}[0]
{\varepsilon }

\newcommand{\scr}[1]
{\mathscr{#1}}

\newcommand{\bb}[0]
{\mathbb}

\newcommand{\vvv}[1]
{\textquoteleft #1\textquoteright\,}

\newcommand{\dl}[0]
{ {{\textnormal{d\hspace{-4.6pt} l}}} }

\begin{document}

	\title{Representation of weak solutions of convex Hamilton-Jacobi-Bellman equations on infinite horizon}
	
	\subtitle{}
	
	\author{Vincenzo Basco}
	
	\institute{Vincenzo Basco \at	
		Department of Electrical \& Electronic Engineering\\
		The University of Melbourne\\
		Victoria 3010\\
		 Australia\\
		 vincenzobasco@unimelb.edu.au, vincenzobasco@gmail.com.
	}

	\maketitle

	\begin{abstract}
		In the present paper it is provided a representation result for the weak solutions of a class of evolutionary Hamilton-Jacobi-Bellman equations on infinite horizon, with Hamiltonians measurable in time and fiber convex.	
		Such Hamiltonians are associated to a $-$ faithful $-$ representation namely involving two functions measurable in time and locally Lipschitz in the state and control. Our results concern to recover a representation of convex Hamiltonians under a relaxed assumption on the Fenchel transform of the Hamiltonian with respect the fiber.	We apply them to investigate uniqueness of weak solutions vanishing at infinity of a class of time dependent Hamilton-Jacobi-Bellman equations, regarded as an appropriate value function of an infinite horizon control problem under state constraints,	assuming a viability condition on the domain of the aforementioned Fenchel transform.
	\end{abstract}

	\section{{I{ntroduction}}}
	
	In this paper we address the Hamilton-Jacobi-Bellman (HJB) equation on infinite horizon
	\equazioneref{def_weak_sol}{
		-V_t+H(t,x,-D_x V)=0\quad\ttnn{on }(0,+\infty) \times \scr O,
	}
	where $H:\bb R^+\times \bb R^n\times \bb R^n\ra \bb R$ is the Hamiltonian, $\scr O \subset \bb R^n$ is an open subset, and $D_x$ stands for the gradient with respect to the space variable.
	The notion of \textit{weak} $-$ or \textit{viscosity} $-$ solution to a first-order partial differential equation to study stationary and evolutionary HJB equations is due to Crandall, Evans, and Lions \cite{crandalllionsevans1984someproperties,lios1982generalizedsolutions}. Using superdifferentials and subdifferentials of continuous functions, they obtained existence and uniqueness results in the class of continuous functions for Cauchy problems associated to HJB equations, when the Hamiltonian is continuous, by means of concept of sub/super-solution. Barles and Souganidis \cite{barles1984existenceresults,souganidis1985existence} extended the existence results to a large class of continuous Hamiltonians. As matter of fact, such notion of solution turns out to be quite unsatisfactory for HJB equations arising in control theory and the calculus of variations $-$ cfr. \cite{vinter00} for further discussions. In fact, the value function, that is a viscosity solution of HJB equation, loses the differentiability property $-$ even in the absence of state constraints $-$ whenever multiple optimal solutions are present at the same initial datum or when additional state constraints arise. Further, the definition of solution can be stated equivalently in terms of \vvv{normals} to the epigraph and the hypograph of the solution. But, when the dynamics is only measurable in time  such equivalence may fail to be true.
	
	Nevertheless, the study of uniqueness of weak solutions can be carried out by using the definition of solution from \cite{frankowskaplaskrze1995measviabth}. In order to deal with Hamiltonians which are measurable in time, Ishii \cite{ishii1985hbjediscontinuous} proposed a new notion of weak solution in the class of continuous functions, proving the existence and uniqueness in the stationary case, and, for the evolutionary case, on $(0,+\infty)\times \bb R^n$. 
	The continuously differentiable test functions needed to define such solutions are more complex, involving in addition some  integrable mappings. This yields an existence result for weak solutions.
	Since uniqueness results for viscosity solutions of the Bellman equation
	\eee{
		-V_t+\sup_{u\in \bb R^m}  \graffe{ \ps{f(t,x,u)}{-D_x V} -\ell(t,x,u)   }=0\quad\ttnn{on }(0,+\infty) \times \scr O,
	}
	assert that the weak solutions are represented as the value function of the control problem associated to the couple $(f,\ell)$ where $f$ is the dynamics and $\ell$ the Lagrangian,	one may ask the possibility for the viscosity solution of the HJB in \rif{def_weak_sol} to be represented as the value function of an appropriate optimal control control problem under state constraints. In the compact time case, this viewpoint was investigated by Ishii in \cite{ishii1985representationconvex} for the convex case, providing H\"older continuous representation, and in \cite{ishii1988representation} for Hamiltonian non necessarily convex. In this latter work, the lagrangian $\ell$ is merely continuous and the space of control is infinite dimensional. On the other hand, in \cite{rampazzo2005faithful} the author construct a faithful representation, Lipschitz continuous in the state and control. Frankowska and Sadrakyan \cite{frankowskasedrak2014stablerephami} investigated faithful representations of Hamiltonians that are measurable in time, giving sharp results on the Lipschitz constants of faithful representations, and the stability of the faithful representation, key property to show convergences results of value functions. However, in \cite{misztela2019repre} the author, under some weaker assumptions and assuming the boundedness from above of the Fenchel transform $H^*(t,x,.)$ on its domain, constructed a faithful representation $(f,\ell)$ showing, for the finite horizon setting, the equivalence between the calculus of variation problem and the optimal control problem associated to the representation $(f,\ell)$ for the free-constrained setting $\scr O=\bb R^n$.
	
	Unfortunately, when addressing state constrained problems, i.e. $\scr O\neq \bb R^n$, the usual assumptions on the Hamiltonian may be insufficient to derive existence and uniqueness results for the HJB equations, even for finite horizon problem. In the framework of control problems, Soner \cite{Soner86a} proposed a controllability assumption $-$ called \textit{inward pointing condition} $-$ to investigate the continuity of the value function and the uniqueness of viscosity solutions of an autonomous control problem. However, such a property cannot be used for sets with nonsmooth boundaries and boundedness assumptions on $\scr O$ may be quite restrictive for many applied models. To allow for nonsmooth boundaries, Ishii and Koike \cite{ishiikoike1996newformulationstateconstrained} generalized Soner's condition in the setting of infinite horizon problems and continuous solutions $-$ cfr. also \cite{bascofrankowska2018lipschitz}. To deal with discontinuous solutions, Ishii \cite{ishii1992perron} introduced the concept of \textit{lower} and \textit{upper semicontinuous envelopes} of a function, proving that the upper semicontinuous envelope of the value function of an optimal control problem is the largest upper semicontinuous subsolution and its lower semicontinuous envelope is the smallest lower semi-continuous supersolution. This approach, however, does not ensure the uniqueness of $-$ weak $-$ solutions of the HJB equation. On the other hand, the upper semicontinuous envelope does not have any meaning in optimal control theory while dealing with minimization problems $-$ the lower semicontinuous envelope determines the value function of the relaxed problem. Barron, Frankowska, and Jensen \cite{barronjensen1991optimalcontrolsemicontinuous,frankowska1993lowersemicontinuous} developed a different concept of  solutions for the HJB equation associated to constraint-free Mayer optimal control problems, with a discontinuous cost. In this approach only subdifferentials are involved. While investigating infinite horizon problems, in the early work \cite{bascofrank2019hjbinfhor}, the  merely measurable case, it became clear that, in order to get uniqueness,  it is convenient to replace subdifferentials by normals to the epigraph of solutions.  Such {\textquoteleft{geometric}\textquoteright} definition of solution avoids using test functions  and  allows to have a unified approach to both the  continuous and measurable case.
	
	The contribution of the present paper is to give a representation and an uniqueness theorem for weak solutions $-$ in the sense of definition given in Section \ref{sec_inf_hor_control_prob} $-$ of non-autonomous HJB equations \rif{def_weak_sol}, with Hamiltonians time measurable and convex in the fiber. We prove, assuming the relaxed assumption of local boundedness from above of $H^*(t,x,.)$ on its domain, a faithful representation result $-$ cfr. Section 3 $-$ of convex Hamiltonians in order to recover, under a backward viability assumption on the domain of the Fenchel transform $H^*(t,x,.)$, the uniqueness of weak solutions in the class of vanishing functions at infinity
	\equazioneref{vanishing_cond}{
		\lim_{t\ra +\infty}\,\sup_{x\in \ttnn{dom}\,V(t,.)}\modulo{V(t,x)}=0.
	}
	
	The outline of this paper is as follows. In Section 2 we recall some basic concept and result in non-smooth analysis. The Sections 3 and 4 are devoted to the parametrization of set-valued maps and the representation of convex Hamiltonians, respectively. In Section 5 we state the main result of this paper, showing an uniqueness theorem for weak solution of HJB on infinite horizon with vanishing condition at infinity \rif{vanishing_cond} and a representation result of such $-$ weak $-$ solutions as the value function of an appropriate infinite horizon optimal control problem under state constraints.

	N\textsc{\footnotesize{{OTATIONS:}}} $\bb R^+$, $|\,.\,|$, and $\ps{.}{.}$ stands for the set of all non-negative real numbers, the Euclidean norm, and the scalar product, respectively. Let $E\subset \bb R^k$ be a subset and $x\in \bb R^k$. The Euclidean distance between $x$ and $E$ and the closed ball in $\bb R^k$ of radius $r>0$ and centered at $x$ are denoted, respectively, by $d(x,E)$ and $B(x,r)$ ($\bb B:=B(0,1)$). $\ttnn{cl }E$, $\ttnn{int }E$, $\ttnn{bdr }E$, $E^c$, and $\ttnn{co }E$ stands, respectively, for the closure, the interior, the boundary, the complement, and the convex hull of $E$ ($\overline{\ttnn{co}}\,E:=\ttnn{cl }\ttnn{co }E$). We put $\norm{E}:=\sup_{k\in E}|k|\in \bb R^+\cup \{+ \infty\}$. $\scr C^m$ stands for the family of all non-empty closed convex subsets in $\bb { R } ^ { m }$  and we write $J\in \scr C^m_b$ if $J$ is bounded and $J\in \scr C^m$. $\mu$ denotes the Lebesgue measure.

	\section{{P{reliminaries on non-smooth analysis}}}\label{sec_preliminaries}

	The \textit{negative polar cone} of a non-empty subset $C\subset \bb R^k$ is the set defined by 
	
	$
		C^-:=\graffe{p\in \bb R^k\,|\,\ps{p}{c}\mineq 0 \ttnn{ for all } c\in C}.
	$
	
	\noindent The \textit{positive polar cone} of $C$ is the set defined by $C^+:=-C^-$. Let $D\subset \bb R^n$ be non-empty and $x\in \ttnn{cl }{D}$. The \textit{contingent cone} to $D$ at $x$ is the set defined by
	
	$
	T_D(x):=\{v\in \bb R^n\,|\,  \liminf_{h\ra 0+} \dfrac{d_D(x+hv)}{h}=0\}.
	$
	
	\noindent The \textit{limiting normal cone} to $D$ at $x$, written $N_D(x)$, is the closed set of all $p\in \bb R^n$ such that $\liminf_{y\ra_D x} d_{T_D(y)^-}(p)=0$ and it is known that  $N_D(x)^-\subset T_D(x)$, provided $D$ is closed. 
	If $D$ is convex, then $N_D(x)$ is also called \textit{normal cone} and holds (cfr. \cite{vinter00})
	\equazioneref{cono_normale_per_convessi}{
		p\in N_D(x) \Longleftrightarrow \ps{p}{y-x}\mineq 0\;\; \forall  y\in D.
	}
	We denote for any $r>0$
	
	$
		N^{r}_D(x):=\{  p\in \bb R^n\,|\, p\in {\overline{\ttnn{co}}}\,N_D(y),\,y\in (\ttnn{bdr } D)\cap B(x,r),\, |p|=1 \}.
	$
	
	\noindent Now, assume that $D$ is closed. A vector \(p \in \bb{R}^{n}\) is called a \textit{proximal normal} to $D$ at $x$ if there exists $\lambda>0$ such that $d_{D}(x+\lambda p)=\lambda|p|$, i.e.,
	\equazioneref{B_subset_D_complement}{
		\ttnn{int }B(x+\lambda p,\lambda|p|) \subset D^{c}.
	}
	We note that, since $D$ is closed, for any \(x \in D\) the set of all proximal normals is non-empty and it reduces to the singleton \(\{0\}\) at any interior point of \(D\).

	Let $\varphi:\bb R^k\ra \bb R\cup \{\pm \infty\}$ be an extended real function. We write $\ttnn{dom }\varphi$, $\ttnn{epi }\varphi$, $\ttnn{hypo }\varphi$, and  $\ttnn{graph }\varphi$ for the \textit{domain}, the \textit{epigraph}, the \textit{hypograph},  and the \textit{graph} of $\varphi$, respectively.	We recall that $\varphi$ is said to be measurable if $\varphi^{-1}(\graffe{+\infty})$, $\varphi^{-1}(\graffe{-\infty})$, and $\varphi^{-1}(I)$ are measurable for any Borel subset $I\subset \bb R$. The function $\varphi$ is said to be \textit{proper} if $\ttnn{dom }\varphi\neq \emptyset$ and $\varphi$ never attain $-\infty$ values. We reacall that the \textit{contingent epiderivative/hypoderivative}, in direction $u\in \bb R^k$, of $\varphi$ at $x\in \ttnn{dom}\,\varphi$ are, respectively, defined by
	
	$D_\uparrow \varphi(x)(u):=\liminf_{h\ra 0+,\,u'\ra u}\dfrac{\varphi(x+hu')-\varphi(x)}{h},$
	
	$D_\downarrow \varphi(x)(u):=\limsup_{h\ra 0+,\,u'\ra u}\dfrac{\varphi(x+hu')-\varphi(x)}{h}.$
	
	\noindent The Fr\'echet \textit{subdifferential/superdifferential} of $\varphi$ at $x\in \ttnn{dom}\,\varphi$ are, respectively, defined by
	
	$\partial_- \varphi (x):=\{p\in \bb R^k\,|\, \liminf_{y\ra x}\dfrac{\varphi(y)-\varphi(x)-\ps{p}{y-x}}{\modulo{y-x}}\mageq 0\},$
	
	$\partial_+ \varphi (x):=\{p\in \bb R^k\,|\, \limsup_{y\ra x}\dfrac{\varphi(y)-\varphi(x)-\ps{p}{y-x}}{\modulo{y-x}}\mineq  0\}.$

	\noindent The following result is well known (cfr. \cite[Theo. 1]{rockafellar1981proximal} and \cite[Prop. 8.12]{rockafellar2009variational}).
	\begin{lemma}[\cite{rockafellar1981proximal,rockafellar2009variational}]\label{lemma2}
		Assume that $\varphi$ is lower semicontinuous and convex.
		
		Then, for any $x \in \ttnn{dom } \varphi$:
		\margin{\hspace{-6mm}\ttnn{(i)} $\partial_+ \varphi({x})= \partial \varphi({x}):=\{p \in \bb R^k\,|\, \varphi(y) \mageq \varphi({x})+\langle p, y-{x}\rangle \text { for all } y\in \bb R^k\}$ and it is called \textit{subdifferential} (in the sense of convex analysis) of $\varphi$ at $x$;
			
		}
		
		\margin{\hspace{-7mm}\ttnn{(ii)} for any $( p , 0 ) \in N _ { \ttnn {epi }\varphi } ( x , \varphi( x ) )$	there exist two sequences $x_i\in \ttnn{dom }\varphi$ and $\left( p _ { i } , q _ { i } \right) \in N _ { \ttnn {epi } \varphi} \left( x _ { i } , \varphi\left( x _ { i } \right) \right)$ such that $q_i<0$ for all $i\in \bb N$ and $	\left( x _ { i } , \varphi\left( x _ { i } \right) \right) \ra ( x , \varphi( x ) )$, $  (p_i,q_i)\ra (p,0)$.
			
		}

	\end{lemma}

	The Fenchel \textit{transform} (or \textit{conjugate}) of $\varphi$, written $\varphi^*$, is the extended real function $\varphi^*:\bb R^k\ra \bb R\cup \{\pm\infty\}$ defined by $\varphi^*(v):=\sup_{p\in \bb R^k}\{\ps{v}{p}-\varphi(p)\}$. The following results are known (cfr. \cite[Theo. 11.1, 11.3]{rockafellar2009variational}).
	\begin{lemma}[\cite{rockafellar2009variational}]\label{lemma1}
		Assume that $\varphi$ is proper, lower semicontinuous, and convex.
		
		Then $\varphi^*$ is a proper lower semicontinuous convex function, $(\varphi^*)^*=\varphi$, $\ttnn{dom }\varphi^*$ is convex, and	for all $p,\,v\in \bb R^k$ it holds that $p \in \partial \varphi^ { * } ( v )\Longleftrightarrow v \in \partial \varphi( p )  \Longleftrightarrow  \varphi( p ) + \varphi^ { * } ( v ) = \langle v , p \rangle.$
		
	\end{lemma}

	\section{{P{arametrization of set-valued maps}}}
	
	We recall that the \textit{extended Hausdorff distance} between \(J,\, K\in \scr C^m\) is defined by
	
	$\dl (J, K) :=\max\; \{\sup _{x \in K} d(x, J), \sup _{x \in J} d(x, K)\}\in \bb R \cup \graffe{+\infty}.$
	
	\noindent Notice that $\dl(J,K)<+\infty$ for any $J,K\in \scr C^m_b$. Next we state a result on Lipschitz parametrization of convex sets (cfr. \cite[Chapter 9]{aubin2009set}).
	
	\begin{lemma}[\cite{aubin2009set}]\label{lemma3}	
		Let $P : \bb { R } ^ { m } \times \scr C^m \rightsquigarrow \scr C^m$ be the projection map defined by $P ( u , J ) := J \cap B (u , 2 d ( u , J ) ).$

		Then 	${\dl  } ( P ( u , J ) , P ( v , K) ) \leq 5 (  { \dl  } ( J , K) + | u - v | )$ for all $J , K \in \scr C^m$ and all $u,v \in \bb { R} ^ { m }$.	

	\end{lemma}

	\noindent In the following, we consider the map $S_m:\scr C^m_b\ra \bb R^m$ defined by
	
	$
		S_m(J):= \dfrac{1}{\mu(\bb B)}\int_{\bb B} \ttnn{pr}(\partial \sigma_J(p))\,\mu(dp),
	$
	
	\noindent where: for any $K\in \scr C^m$, $\ttnn{pr}(K)$ stands for the projection of $0\in \bb R^m$ onto $K\in  \scr C^m$, i.e., the element in $K$ with minimal norm; for any $J\in \scr C^m_b$, $\sigma_J\ccd$ denotes the support function of $J$, that is $\sigma_J(p):=\max_{q\in J}\ps{p}{q}$. The function $S_m\ccd$ is called \textit{Steiner map (or Steiner selection)} and the following result is well known (cfr. \cite[Theorem 9.4.1]{aubin2009set}).

	\begin{lemma}[\cite{aubin2009set}]
		The function $S_m\ccd$ is $m$-Lipschitz continuous on $\scr C^m_b$ with respect the Hausdorff distance and satisfies
		\equazioneref{S_in_J}{
			S_m(J)\in J\qquad \forall J\in \scr C^m_b.
		}
	\end{lemma}
	\begin{remark}\rm

		We notice that \rif{S_in_J} follows immediately from the properties of Fenchel transform and the definition of subdifferential. Indeed, fix $J\in \scr C^m_b$ and let $p\in \bb B$. Define $\psi(.)=\psi_p\ccd:=\sigma_J(.+p)$. The function $\psi$ is proper convex. From Lemma \ref{lemma2}-(i), it follows that
		\equazioneref{passo_a}{
			\partial \sigma_J(p)=\partial \psi(0)=\ttnn{arg}\min \psi^*.
		}
		So, $\psi^*(q)=\sup_{y\in \bb R^m} \{\ps{y}{q}-\sigma_J(y+p)\}=-\ps{p}{q}$, if $q\in J$, and $+\infty$ otherwise. From \rif{passo_a}, we have $\partial \sigma_J(p)={\ttnn{arg}\max}_{q\in J}\ps{p}{q}$, and, by arbitrariness of $p$, we conclude $\ttnn{pr}(\partial \sigma_J(p))\in J$ for all $p\in \bb B$. Since $\dfrac{1}{\mu(\bb B)}\int_{\bb B}J\mu(dp)=J$, we get \rif{S_in_J}.
	\end{remark}

	Next, we state the main result of this section on parametrization of convex sets following the main ideas of those discussed in \cite[Chapter 9]{aubin2009set} (cfr. also the literature therein) and providing sharper conditions.

	\begin{theorem}\label{teo_rep_convex_sets}
		Let $I$ be a closed interval of $\bb R^+$ and $Q:I\times \bb R^k \rightsquigarrow \bb R^m$ be a set-valued map such that $Q(t,x)\in \scr C^m$ for all $ (t,x)\in I\times \bb R^k$, $Q(.,x)$ is measurable for all $x\in \bb R^k$, and:
		\enualp{

			\item for all $t\in I$ and any $r>0$ there exists $c_r(t)>0$ satisfying	$Q(t,x)\subset Q(t,y)+ c_r(t) |x-y|\bb B$	for all $x,y\in B(0,r)$.
		}
		
		Then, for any set-valued map $\delta:I\times \bb R^k\rightsquigarrow \bb R^m$, with non-empty closed values and ${\delta(.,x)}$ measurable for all $x\in \bb R^k$, there exist two functions \(\phi : I \times \bb R^k \times \bb R^m \ra \bb {R}^{m}\) and $\eta:I \times \bb R^k \ra (0,+\infty)$ satisfying
		\begin{eqnarray}\label{prop_eta}
		\eta(t,x)=
		\begin{cases}
		\norm{\delta(t,x)} &\ttnn{if }\norm{\delta(t,x)}>0\\
		1 &\ttnn{otherwise},
		\end{cases}
		\end{eqnarray}

		and:
		\enurom{
			\item \(\phi(., x, u)\) and $\eta(.,x)$ are measurable for all $x\in \bb R^k,\,u\in \bb R^m$;
			
			\item for any $t\in I$ and any $r>0$
			\eee{
				&|\phi(t, x, u)-\phi(t, {y}, {v})| \mineq 5m( c_r(t) |x-y|+ |\eta(t,x)u-\eta(t,y){v}|)\\
				&\forall x,y\in B(0,r),\, \forall u,v\in \bb R^m;
			}

			\item $\phi(t,x,\bb B)\subset Q(t,x)$ for all $(t,x)\in I\times \bb R^k$;
			
			\item if $\delta(t,x)\neq \{0\}$ and it is bounded, then $Q(t,x)\cap \delta(t,x) \subset \phi(t,x,\bb B)$.
		}
		
		In particular, if $\delta(.,.)\equiv \bb R^m$, then
		\equazioneref{F_uguale_f_2}{
			Q(t,x)=\phi(t, x, \bb R^m)\quad \forall (t,x)\in I\times \bb R^k.
		}

	\end{theorem}
	\begin{proof}

		Assume first that $\delta:I\times \bb R^k\rightsquigarrow \bb R^m$ is the constant set-valued map $\delta(.,.)\equiv \bb R^m$. Next, we prove (i)-(iii). Notice that, from our assumptions, \cite[Theorem 8.2.3]{aubin2009set}, and since the intersection of measurable set-valued maps is measurable, we have that
		\equazioneref{misurabile_P}{
			\forall (x, u)\in \bb R^k\times \bb R^m,\quad t\rightsquigarrow  P(u,Q(t,x)) \ttnn{ is measurable},
		}
		where $P(.,.)$ is the projection map defined in Lemma \ref{lemma3}. Fix $r>0$. From assumption (a) and applying Lemma \ref{lemma3}, for all $t\in I$ and all $x,y\in  B(0,r)$
		\equazioneref{Lip_P}{
			\dl (P(u,Q(t,x)),P(v,Q(t,y)))&\mineq 5(\dl (Q(t,x),Q(t,y))+|u-v|)\\
			&\mineq 5(c_r\rt|x-y|+|u-v|).
		}
		Now, consider the function \(\phi : I \times \bb R^k \times \bb R^m \ra \bb {R}^{m}\) defined by
		
		$
			\phi(t,x,u):=S_m\circ P(u,Q(t,x)).
		$
		
		\noindent Let $t\in I$, $x\in \bb R^k,$ and $u\in \bb R^m$. By \rif{S_in_J} immediately follows that $\phi(t,x,u)\in Q(t,x)$. In particular, (iii) holds. Moreover, let $w\in Q(t,x)$. Since $Q(t,x)\cap B(w,2d(w,Q(t,x)))=\{w\}$, then $\phi(t,x,w)=S_m\circ P(w,Q(t,x))=w$. So, \rif{F_uguale_f_2} is proved. From the $m$-Lipschitz continuity of $S_m\ccd$ and \rif{Lip_P}, it follows that for all $t\in I$,
		\equazioneref{stima_lip_su_phi}{
			|\phi(t,x,u)-\phi(t,y,v)|&\mineq m \dl (P(u,Q(t,x)),P(v,Q(t,y)))\\
			&\mineq 5 m(c_r\rt|x-y|+|u-v|)
		}
		for all $x,y\in B(0,r)$ and all $u,v\in \bb R^m$. Hence, recalling \rif{misurabile_P}, \rif{stima_lip_su_phi}, and the continuity of $S_k\ccd$, (i) and (ii) follows.

		Now, consider a set-valued map  $\delta:I\times \bb R^k\rightsquigarrow \bb R^m$  with non-empty closed values and ${\delta(.,x)}$ measurable for all $x\in \bb R^k$. From \cite[Theorem 8.2.11]{aubin2009set} and since $\delta(.,x)$ is measurable for any $x\in \bb R^k$, we have that the map $t\ra \norm{\delta(t,x)}\in \bb R\cup \{\pm \infty\}$ is measurable for any $x	\in \bb R^k$.  Define, for any $x\in \bb R^k$, the measurable set $\Lambda(x):=\graffe{ t\in I\,|\,\norm{\delta(t,x)}\in \graffe{0,+\infty}}$ and denote for all $(t,x)\in I\times \bb R^k$

		$
			\eta(t,x):=\chi_{\Lambda(x)}(t)+\chi_{\Lambda(x)^c}(t)\cdot\norm{\delta(t,x)},
		$
		
		\noindent where $\chi_V$ is the function that takes values $1$ on the set $V$ and $0$ elsewhere. Then, \rif{prop_eta} holds and, from our assumptions, the map $t\mapsto \eta(t,x)$ is measurable for all $x\in\bb R^k$. Consider the map  \(\phi : I \times \bb R^k \times \bb R^m \ra \bb {R}^{m}\)  defined by
		
		$
			\phi (t,x,u):= S_m\circ P(\eta(t,x)u,Q(t,x)).
		$
		
		\noindent Arguing in the same way as above, the statements \ttnn{(i), (ii)}, and (iii) holds. Next we show (iv). Assume that $\delta(t,x)$ is bounded and $\delta(t,x)\neq \{0\}$. From the definition of $\eta(t,x)$, we have that $\eta(t,x)= \norm{\delta(t,x)}>0$. Now, if $Q(t,x)\cap \delta(t,x) = \emptyset$, then (iv) holds. Otherwise, let $w\in Q(t,x)\cap \delta(t,x)$. Then, there exists $\hat \delta\in [0,\norm{\delta(t,x)}]$ and $|{\hat w}|=1$ satisfying   $w=\hat w\hat \delta$. We have $w=(\hat w\frac{\hat \delta}{\eta(t,x)})\eta(t,x)$. Since $\modulo{\hat w \frac{\hat \delta}{\eta(t,x)}}\mineq 1$, it follows $\phi(t,x,\hat w \frac{\hat \delta}{\eta(t,x)})=w$. Thus, (iv) holds.
	\qed\end{proof}
	\begin{remark}\label{remark_stima_H_star_ell}\rm
		Let $(t,x)\in I\times \bb R^n$. From the proof of Theorem \ref{teo_rep_convex_sets}, it follows that $\dl(Q(t,x)\cap \delta(t,x), \phi(t,x,\bb B))\mineq 10 m\norm{\delta(t,x)}$. Indeed, for any $\gamma\in Q(t,x)\cap \delta(t,x)$ and any $\theta\in \phi(t,x,\bb B)$, we have
		
			$
	|\gamma-\theta|=|\gamma-S_m\circ P(\norm{\delta(t,x)}u,Q(t,x))|
	$
	
	$	
	\quad \quad \quad\hspace{0.35mm}=|S_m\circ P(\gamma,Q(t,x))-S_m\circ P(\norm{\delta(t,x)}u,Q(t,x))|
	$
	
	$
	\quad \quad \quad\hspace{0.35mm}\mineq 5m |\gamma-\norm{\delta(t,x)}u|\mineq 10 m\norm{\delta(t,x)}.
	$
	\end{remark}
	
	From Theorem \ref{teo_rep_convex_sets} we get the following corollary:
	\begin{corollary}
		Assume the assumptions of Theorem \ref{teo_rep_convex_sets} and that for all $x\in \bb R^k$ there exists $r(.,x):I \ra  (0,+\infty)$ measurable such that $Q(t,x)\subset r(t,x)\bb B$ for all $(t,x)\in I\times \bb R^k$. 
		
		Then there exists a function \(\phi : I \times \bb R^k \times \bb R^m \ra \bb {R}^{m}\) such that
		\eee{
			Q(t,x)=\phi(t, x, \bb B)  \quad \forall (t,x)\in I\times \bb R^k,
		}
		and:
		\enurom{
			
			\item \(\phi(., x, u)\) is measurable for all $x\in \bb R^k,\,u\in \bb R^m$;
			
			\item for any $t\in I$ and any $r>0$
			\eee{
				&|\phi(t, x, u)-\phi(t, {y}, {v})| \mineq 5m( c_r(t) |x-y|+ | r(t,x)u- r(t,y){v}|)\\
				&\forall \,x,y\in B(0,r),\,\forall u,v\in \bb R^m.
			}
			
		}
	\end{corollary}
	\begin{proof}
		All the conclusions follows from Theorem \ref{teo_rep_convex_sets} by choosing $\delta(.,.)=r(.,.)\bb B$.
	\qed\end{proof}

	\section{{R{epresentation of convex Hamiltonians}}}

	Let $I$ be a closed interval of $\bb R^+$ and $H:I\times \bb R^n\times \bb R^n\ra\bb R$ be a function such that $t\mapsto H(t,x,p)$ is measurable for any $x,p\in \bb R^n$. We consider the following conditions on $H$:
	
	\margin{\hspace{-6mm}\ttnn{H.1.1}. $p\mapsto H(t,x,p)$ is convex for all $t\in I$ and $x\in \bb R^n$;}
	
	\margin{\hspace{-6mm}\ttnn{H.1.2}. for all $r>0$ there exists $C_r:I\ra \bb R^+$ measurable such that
		
		\vspace{2mm}
		
		$|H(t,x,p)-H(t,y,p)|\mineq  C_r(t)(1+|p|)|x-y|$
		
		\vspace{2mm}
		
		for all $t\in I$, $x,y\in B(0,r)$, and $p\in \bb R^n$;}

	\margin{\hspace{-6mm}\ttnn{H.1.3}. there exists  $\tilde c:I\ra \bb R^+$  measurable such that
		
		\vspace{2mm}
		
		$	|H(t,x,p)-H(t,x,q)|\mineq \tilde c(t)(1+|x|)|p-q|$
		
		\vspace{2mm}
		
		\indent for all $t\in I$ and $x,p,q\in \bb R^n$.
	}

		\noindent In the following, for any $(t,x)\in I\times \bb R^n$, we denote by $H^*(t,x,.):\bb R^n\ra \bb R\cup \{\pm \infty\}$ the Fenchel transform of the function $H(t,x,.)$, we define the set-valued maps $D:I\times \bb R^n\rightsquigarrow \bb R^n$, $Ep:I\times \bb R^n\rightsquigarrow \bb R^{n+1}$, and $Gr:I\times \bb R^n\rightsquigarrow \bb R^{n+1}$ respectively by
	
	$
		D(t,x):=\ttnn{dom } H^*(t,x,.),
	$
	
	$
		Ep(t,x):=\ttnn{epi }H^*(t,x,.),
	$
	
	$
		Gr(t,x):=\ttnn{graph }H^*(t,x,.),
	$
	
	\noindent and we put $\gamma(t,x):=\max\; \{ \, 0\, ,\,  \sup_{q\in D(t,x)} \,H^*(t,x,q)  \} \in \bb R^+\cup \{+ \infty\}$.
	
	We also consider the following condition on the Hamiltonian:

	\margin{\hspace{-6mm}\ttnn{H.1.4}. $\forall(t,x) \in I \times \bb {R}^{n},\; \forall \bar q\in \ttnn{cl }D(t,x),\; \exists \eps>0:$
		
		\vspace{2mm}
		
		$\sup_{q\in D(t,x)\cap B(\bar q,\eps)} \,H^*(t,x,q)<\infty.$}

	\begin{theorem}[Representation]\label{theo_rep_H}
		Assume \ttnn{H.1.1-2} and \ttnn{H.1.4}.
		
		Then there exists a function \(\phi : I \times \bb R^n \times \bb R^{n+1} \ra \bb {R}^{n+1}\),
		\eee{
			\phi(t,x,u)=:(f(t,x,u),\ell (t,x,u))\in \bb R^n\times \bb R,
		}

		satisfying:
		\enurom{
			\item \(f(., x, u)\) and \(\ell(., x, u)\) are measurable for all $x\in \bb R^n$ and $u\in \bb R^{n+1}$;
			\item for all $t\in I$ and $x,p\in \bb R^n$,
			\eee{
				H(t,x,p)=\sup_{u\in \bb R^{n+1}}\,\{\ps{p}{f(t,x,u)}-\ell(t,x,u)\};
			}		
			
			\item for all $t\in I$ and $r>0$,
			\eee{
				&|f(t, x, u)-f(t, {y}, {v})|\mineq 5(n+1)( C_r(t) |x-y|+ |u-{v}|),\\
				&|\ell(t, x, u)-\ell(t, {y}, {v})| \mineq 5(n+1)     ( C_r(t) |x-y|+ |u-{v}|)\\
				&\forall\, x,y\in B(0,r) ,\,\forall u,v\in \bb R^{n+1}.
			}

		}
		
		If, in addition, condition \ttnn{H.1.3} holds, then the statements \ttnn{(ii)}-\ttnn{(iii)} are replaced by the following:
		\enurom{
			\item[\ttnn{(ii)}$'$]  for all $t\in I$ and $x,p\in \bb R^n$,
			\eee{
				H(t,x,p)=\sup_{u\in \bb B}\,\{\ps{p}{f(t,x,u)}-\ell(t,x,u)\};
			}		
			\item[\ttnn{(iii)}$'$]  for all $t\in I$ and $r>0$,
			\eee{
				&|f(t, x, u)-f(t, {y}, {v})|\mineq 5(n+1)( C_r(t) |x-y|+ |\eta(t,x)u-\eta(t,y){v}|),\\
				&|\ell(t, x, u)-\ell(t, {y}, {v})| \mineq 5(n+1)     ( C_r(t) |x-y|+ |\eta(t,x)u-\eta(t,y){v}|)\\
				&\forall x,y\in B(0,r)   ,\,\forall u,v\in \bb R^{n+1},
			}
			
			where $\eta(t,.):=\tilde c(t)(1+|.|)+\gamma(t,.)+|H(t,.,0)|$;
		}
		and moreover:
		\enurom{	
			\item[\ttnn{(iv)}$'$] $D(t,x)=f(t,x,\bb B)$ for all $t\in I$ and $x\in \bb R^n$;

			\item[\ttnn{(v)}$'$]  $Gr(t,x)\subset \phi(t,x,\bb B)$ for all $t\in I$ and $x\in \bb R^n$.
		}

	\end{theorem}
	
	Before to give a proof of Theorem \ref{theo_rep_H}, we show some intermediate results.

	\begin{lemma}\label{lemma1_sec_hamil}
		Assume  \ttnn{H.1.1-2}  and let \((t, x) \in I \times \bb {R}^{n}\).
		
		Then:
		\enurom{
			\item ${D(t, x)}$ is non-empty and convex.
		}
		Moreover, if, in addition, the condition  \ttnn{H.1.3} holds, then:
		\enurom{
			\item[\ttnn{(ii)}] $D(t,x)\subset \tilde c(t)(1+|x|) \bb B.$
		}
		
	\end{lemma}
	\begin{proof}
		The first claim follows immediately from Lemma \ref{lemma1}. The proof of statement (ii) follows in the same way as in \cite{frankowskasedrak2014stablerephami}.	
			\qed\end{proof}

	\begin{lemma}\label{D_Lip}
		Assume \ttnn{H.1.1-2} and \ttnn{H.1.4}.

		Then:	
		\enurom{
			\item \(D(t, x)\in \scr C^n\);

			\item for all $t\in I$ and $r>0$,
			\eee{
				D(t,x)\subset D(t,y)+C_r\rt|x-y|\bb B\quad \forall x,y\in B(0,r).
			}
		}
	\end{lemma}
	\begin{proof}
		Let \((t,x) \in I \times \bb{R}^{n}\). Next we show that $D(t,x)$ is closed. Consider a sequence $q_i\in D(t,x)$ converging to $\tilde q\in \bb R^n$. Since $\tilde q\in \ttnn{cl }D(t,x)$, from assumption \ttnn{H.1.4}, there exist $\eps>0$ and $M>0$ such that $|H^*(t,x,q_i)|\mineq M$ for all $q_i\in B(\tilde q,\eps)$. Hence, since the Fenchel transform $q\mapsto H^*(t,x,q)$ is lower semicontinuous (cfr. Lemma \ref{lemma1}), $M\mageq \liminf_{i\ra +\infty} H^*(t,x,q_i)\mageq H^*(t,x,\tilde q)$. So, $\tilde q\in D(t,x)$, and recalling Lemma \ref{lemma1_sec_hamil}, the assertion (i) is  proved.
		
		Now, to show (ii), suppose by contradiction that there exist $t\in I$, $r>0$, $x,y\in B(0,r)$, $w\in D(t,x)$, and $\eta>C_r(t)$ such that
		
		$
			D(t,y)\cap B(w,\eta|x-y|)= \emptyset.
		$
		
		\noindent We divide the proof into three steps.
		
		{S\textsc{tep 1}}: Applying Lemma \ref{lemma1_sec_hamil} and \ttnn{(i)}, the set $D(t,x)$ is closed and convex. Let $\bar q\in D(t,y)$ be the projection of $w$ onto $D(t,y)$ and put $z:=(w-\bar q)/|w-q|$. We have that $z$ is a proximal normal to $D(t,y)$ at $\bar q$, i.e., there exists $\bar \lambda>\eta |x-y|$ such that $d_{D(t,y)}(\bar q+\bar \lambda z)=\bar \lambda$. Consider the hyperplane $\{\xi\in \bb R^n\,|\,\ps{z}{\xi}= \ps{z}{\bar q}\}$. Since $D(t,y)$ is convex and $z$ is a proximal normal, we have that $D(t,y)\subset \{\xi\in \bb R^n\,|\,\ps{z}{\xi}\mineq \ps{z}{\bar q}\}$.  Moreover, from \rif{B_subset_D_complement}, $B(w,\eta|x-y|)\subset \ttnn{int }B(\bar q+\bar \lambda z,\bar \lambda)\subset D(t,y)^c$, and we get
		\equazioneref{Hahn_Banach}{
			\ps{z}{q}\mineq \ps{z}{\bar q}< \ps{z}{w+\eta|x-y|h} \quad \forall q\in D(t,y),\,\forall h\in \bb B.
		}
		Notice that, applying \cite[Proposition 4.2.9]{vinter00}, $z\in N_{D(t,y)}(\bar q)$. Hence, using \rif{cono_normale_per_convessi}, we have $(z,0)\in N_{\ttnn{epi }H^*(t,y,.)}(\bar q, H^*(t,y,\bar q))$.

		{S\textsc{tep 2}}:  From Step 1 and applying Lemma \ref{lemma2}-(ii), consider two sequences $w_i\in \ttnn{dom }H^*(t,y,.)$ and $(p_i,q_i)\in N_{\ttnn{epi }H^*(t,y,.)}(w_i,H^*(t,y,w_i))$, with $q_i<0$, satisfying
		\equazioneref{convergenze}{
			(p_i,q_i)\ra (z,0),\quad(w_i,H^*(t,y,w_i))\ra (\bar q, H^*(t,y,\bar q)).
		}
		So, $(p_i/|q_i|,-1)\in N_{\ttnn{epi }H^*(t,y,.)}(w_i,H^*(t,y,w_i))$ for all $i\in \bb N$. We conclude that $p_i/|q_i|\in \partial H^*(t,y,.)(w_i)$ for all $i\in \bb N$. Thus, from Lemma \ref{lemma1},  
		\equazioneref{H_piu_H_star_uguale_prod_scalare}{
			H(t,y,p_i/|q_i|)+H^*(t,y,w_i)=\ps{w_i}{p_i/|q_i|}\quad \forall i\in \bb N.
		}
		
		{S\textsc{tep 3}}:  Using \rif{convergenze} and \rif{Hahn_Banach} with $h=-z/|z|$, we can assume that $\ps{w_i}{p_i}< \ps{w}{p_i} -\eta|x-y|$ for all large $i\in \bb N$.
		Hence, from assumption {H.1.2} and recalling that $w\in D(t,x)$, we get for all large $i\in \bb N$
		
		$\ps{w_i}{p_i} - \left| q _ { i } \right| H ( t , y , { p _ { i } } /{ \left| q _ { i } \right| })$
		
		$< \ps{w}{p_i} -\eta|x-y| - |q_i|H(t,x,p_i/|q_i|)+(|q_i|+|p_i|)C_R(t)|x-y|$
		
		$=|q_i|(\ps{w}{p_i/|q_i|}-H(t,x,p_i/|q_i|))+((|q_i|+|p_i|)C_r(t)-\eta)|x-y|$
		
		$\mineq |q_i|H^*(t,x,w)+((|q_i|+|p_i|)C_r(t)-\eta)|x-y|.$

		\noindent So, by \rif{H_piu_H_star_uguale_prod_scalare}, for all large $i\in \bb N$
		\equazioneref{pre_conclusione}{
			|q_i|H^*(t,y,w_i)<|q_i|H^*(t,x,w)+((|q_i|+|p_i|)C_r(t)-\eta)|x-y|.
		}
		From assumption \ttnn{H.1.4}, the lower semicontinuity of $H^*(t,y,.)$, and since $w_i\ra \bar q$, the sequence $\{H^*(t,y,w_i)\}_{i\in \bb N}$ is bounded. Then, using again \rif{convergenze}, and passing to the lower limit as $i\ra +\infty$ in \rif{pre_conclusione} we get $0\mineq (|z|C_r-\eta)|x-y|$. Since $|z|=1$, $0\mineq C_r\rt-\eta$, and a contradiction follows.
	\qed\end{proof}

	\begin{lemma}\label{lemma_lip_Ep}
		Assume \ttnn{H.1.1-2} and \ttnn{H.1.4}.
		
		Then:
		\enurom{
			\item $t\rightsquigarrow Ep(t,x)$ is measurable for all $x\in \bb R^n$;
			
			\item $Ep(t,x)\in \scr C^{n+1}$ for all $t\in I$ and $x\in \bb R^n$;
			
			\item\label{Lip_Ep} for all $t\in I$ and $r>0$,
			\eee{
				Ep(t,x)\subset Ep(t,y)+C_r\rt|x-y|(\bb B\times [-1,1])\quad \forall x,y\in B(0,r).
			}
			
		}
	\end{lemma}
	\begin{proof}
		We notice that, from \ttnn{H.1.1-2}, the lower semicontinuity and the convexity of Fenchel transform (cfr. Lemma \ref{lemma1}), it follows immediately that, for any $t\in I$ and $x\in \bb R^n$, the set-valued map $s\rightsquigarrow Ep(s,x)$ is measurable and $Ep(t,x)$ is non-empty, closed, and convex. So, (i) and (ii) holds. 
		
		Now, we show \ref{Lip_Ep}. Fix $x,y\in \bb R^n$, $t\in I$, and consider $(q,\lambda)\in Ep(t,x)$. Without loss of generality we may assume that $C_r(t)|x-y|\neq 0$. We claim the following: there exists $w\in D(t,y)$ satisfying $(w,q+C_r(t)|x-y|)\in Ep(t,y)$. Indeed, from assumption H.1.2, we have that $H(t,x,p)\mineq H(t,y,p)+  C_r(t)(1+|p|)|x-y|$ for all $p\in \bb R^n$, and, from the definition of Fenchel transform,	 
		\equazioneref{passo_1_}{
			\tonde{H(t,y,.)+  C_r(t)(1+|.|)|x-y|}^*(\tilde q)\mineq H^*(t,x,\tilde q)\quad \forall \tilde q\in \bb R^n.
		}	 
		
		\noindent Now, define $h(.):=-C_r(t)|x-y|$ on $B(0,C_r(t)|x-y|)$ and $+\infty$ elsewhere. Notice that $h\ccd$ is a proper lower semicontinuous convex function and $h^*(.)=C_r(t)(1+|.|)|x-y|$. Notice that, since $\ttnn{dom }H(t,y,.)=\bb R^n$ and from assumption \ttnn{H.1.1}, applying Lemma \ref{lemma1} we get $(H^*(t,y,.))^*=H(t,y,.)$. So, for all $z\in \bb R^n$
		
		$\tonde{\,\inf_{q_1\in \bb R^k} H^*(t,y,q_1)+h(.- q_1)}^*(z)$
		
		$:=\sup_{q_2\in \bb R^k} \graffe{\ps{q_2}{z} - \inf_{q_1\in \bb R^k} \graffe{H^*(t,y,q_1)+h(q_2-q_1) } }$
		
		$=\sup_{q_2\in \bb R^k} \graffe{\ps{q_2}{z} +\sup_{q_1\in \bb R^k} \graffe{-H^*(t,y,q_1)-h(q_2-q_1) } } $
		
		$=\sup_{q_2,\,q_1\in \bb R^k} \{\ps{q_2}{z} -H^*(t,y,q_1)-h(q_2-q_1) \} $
		
		$=\sup_{q_1\in \bb R^k} \graffe{  \ps{q_1}{z} -H^*(t,y,q_1)  +\sup_{ q_2\in \bb R^k} \graffe{\ps{q_2-q_1}{z} -h(q_2-q_1) } }$
		
		$=H(t,y, z)+C_r(t)(1+|z|)|x-y|.$	
		
		\noindent Since $\ttnn{dom }h=B(0,C_r(t)|x-y|)$ and the function
		
		$z\mapsto\tonde{\,\inf_{ q_1\in \bb R^k} H^*(t,y, q_1)+h(z-  q_1)}$
		
		\noindent is proper lower semicontinuous and convex, passing to the Fenchel transform and using Lemma \ref{lemma1} we deduce
		
		${\,\inf_{q_1\in \bb R^k}\, H^*(t,y, q_1)+h(q- q_1)}$
		
		$={\,\inf_{q_1\in B(q,C_r(t)|x-y|)} \,H^*(t,y, q_1)-C_r(t)|x-y| }$
		
		$=(H(t,y, .)+C_r(t)(1+|.|)|x-y|)^*(q).$

		\noindent Thus, from \rif{passo_1_}, there exists $w\in B(q,C_r(t)|x-y| )$ satisfying
		
		$
			H^*(t,y,w) -C_r(t)|x-y|   \mineq     H^*(t,x, q).
		$
		
		\noindent Hence, the claim follows. Now, applying Lemma \ref{D_Lip}-(ii), $|q-w|\mineq C_r(t)|x-y|$ because $q\in D(t,x)$ and $w\in D(t,y)$. Finally,	since
		
		$	(q,\lambda)=(w,\lambda+C_r(t)|x-y|)+C_r(t)|x-y|\tonde{\frac{q-w}{C_r(t)|x-y|},-1},$

		\noindent the statement (iii) follows by the arbitrariness of $(q,\lambda)\in Ep(t,x)$. \qed\end{proof}
	\begin{proposition}\label{prop_function_phi}
		Assume \ttnn{H.1.1-2} and \ttnn{H.1.4}.
		
		Then there exists a function \(\phi : I \times \bb R^n \times \bb R^{n+1} \ra \bb {R}^{n+1}\),
		\eee{
			\phi(t,x,u)=:(f(t,x,u),\ell (t,x,u))\in \bb R^n\times \bb R,
		}
		satisfying
		\equazioneref{F_uguale_f}{
			Ep(t, x)=\phi(t, x, \bb R^{n+1}) \quad \forall (t,x)\in I\times \bb R^n,
		}
		and:
		\enurom{

			\item \(\phi(., x, u)\) is measurable for all $x\in \bb R^n$ and $u\in \bb R^{n+1}$;
			
			\item for all $t\in I$ and $r>0$
			\eee{
				&|\phi(t, x, u)-\phi(t, {y}, {v})|\mineq 5(n+1)( C_r(t) |x-y|+ |u-{v}|)\\
				&\forall x,y\in B(0,r),\, \forall u,v\in \bb R^{n+1}.
			}

		}
		If, in addition, \ttnn{H.1.3} holds, then 
		\eee{
			Gr(t,x)\subset \phi(t,x,\bb B)\quad \forall (t,x)\in I\times \bb R^n,
		}
		the statement \ttnn{(ii)}  is replaced by the following:
		\enurom{
			\item[\ttnn{(ii)}$'$]  for all $t\in I$ and $r>0$
			\eee{
				&|\phi(t, x, u)-\phi(t, {y}, {v})|\mineq 5(n+1)( C_r(t) |x-y|+ |\eta(t,x)u-\eta(t,y){v}|)\\
				&\forall\,x,y\in B(0,r)    ,\,\forall u,v\in \bb R^{n+1},
			}
			where $\eta(t,.):=(\tilde c\rt+C_r\rt)(1+|.|)+|H(t,0,0)|+\gamma(t,.)$;
		}
		and moreover:
		\enurom{
			\item[\ttnn{(iii)}$'$] $D(t,x)=f(t,x,\bb B)$ for all $t\in I$ and $x\in \bb R^n$.
					}

	\end{proposition}
	\begin{proof}
		Let $t\in I$ and $r>0$. Notice that, if \ttnn{H.1.3} holds, then $\gamma(t,x)<+\infty$ for any $x\in B(0,r)$. Furthermore, for all $x\in B(0,r)$ and $v\in D(t,x)$
		
		$
			-H(t,x,0)\mineq H^*(t,x,v)\mineq \gamma(t,x).
		$
		
		\noindent We get $|(v,H^*(t,x,v))|\mineq \tilde c\rt(1+|x|)+\gamma(t,x)+|H(t,x,0)|\mineq \tilde c\rt(1+|x|)+C_r\rt(1+|x|)+|H(t,0,0)|+\gamma(t,x)$ for all $x\in B(0,r)$. So
		\equazioneref{stima_norma_delta}{
			\norm{Gr(t,x)}\mineq (\tilde c\rt+C_r\rt)(1+|x|)+|H(t,0,0)|+\gamma(t,x)
		}
		for all $x\in B(0,r)$. Hence, the conclusions follows immediately from Theorem \ref{teo_rep_convex_sets} and Lemma \ref{lemma_lip_Ep}, with $Q(t,x)=Ep(t,x)$ and $\delta(t,x)=Gr(t,x)$.
	\qed\end{proof}

	Next, we give a proof of Theorem \ref{theo_rep_H}.
	
	\noindent {\it Proof of Theorem \ref{theo_rep_H}}
		Consider the function $\phi=(f,\ell)$ of Proposition \ref{prop_function_phi}. The statements (i) and (iii) follows from Proposition \ref{prop_function_phi}. Next we show (ii). 
		
		Fix $t\in I$, $x\in \bb R^n$, and $p\in \bb R^n$. Recalling that $(H^*(t,x,.))^*=H(t,x,.)$, from \rif{F_uguale_f} it follows that for any $u\in \bb R^{n+1}$ the pair $(f(t,x,u),\ell(t,x,u))$ lays in $Ep(t,x)$, i.e.,
		\equazioneref{H_star_min_ell}{
			H^*(t,x,f(t,x,u))\mineq \ell(t,x,u).
		}
		So, for any $u\in \bb R^{n+1}$
		
		$\ps{p}{f(t,x,u)}-\ell(t,x,u)$
		
		$\mineq \ps{p}{f(t,x,u)}-H^*(t,x,f(t,x,u))$
		
		$\mineq \sup_{v\in \bb R^{n+1}}\{\ps{p}{v}-H^*(t,x,v)\}=H(t,x,p)$.
		
		\noindent Then, by arbitrariness of $u\in \bb R^{n+1}$, we get
		\eee{
			\sup_{u\in \bb R^{n+1}}\{\ps{p}{f(t,x,u)}-\ell(t,x,u)\}\mineq H(t,x,p).
		}
		On the other hand, let $v\in D(t,x)$. Since $(v,H^*(t,x,v))\in Ep(t,x)$, from \rif{F_uguale_f}, there exists $w\in \bb R^{n+1}$ such that $(v,H^*(t,x,v))=(f(t,x,w),\ell(t,x,w))$. So, $\ps{p}{v}-H^*(t,x,v)=\ps{p}{f(t,x,w)}-l(t,x,w)\mineq \sup_{u\in \bb R^{n+1}}\{ \ps{p}{f(t,x,u)}-\ell(t,x,u) \}$. Hence,
		\eee{
			H(t,x,p)&=\sup_{v\in D(t,x)}\{ \ps{p}{v}-H^*(t,x,v) \}\\
			&\mineq\sup_{u\in \bb R^{n+1}}\{ \ps{p}{f(t,x,u)}-\ell(t,x,u) \}.
		}
		The last statements, assuming that \ttnn{H.1.3} holds, can be obtained with the same arguments as above using Proposition \ref{prop_function_phi}.
	\qed

	\begin{remark}\label{remarks_sec_rep}\rm
		Let $t\in I$ and $x\in \bb R^n$.
		\enualp{
			\item Assumption \ttnn{H.1.4} is weaker than (H4) in \cite[p. 33]{frankowskasedrak2014stablerephami} and \cite[p. 869]{rampazzo2005faithful}.

			\item 	The condition in \ttnn{H.1.4} is equivalent to require that $H^*(t,x,.)$ is locally bounded on its domain. If $H(t,x,.)$ is globally Lipschitz, then one can show that \ttnn{H.1.4} is equivalent to assume that $H^*(t,x,.)$ is bounded on its domain.
			
			\item 	We would like to underline that, if $H(t,x,.)$ is convex, then, from Lemma \ref{lemma1_sec_hamil}, the domain $D(t,x)$ of the Fenchel transform $H^*(t,x,.)$ turn out to be bounded under the global Lipschitz assumption \ttnn{H.1.3}. In particular, we notice that the Lipschitz condition imply the sublinear growth of $H(t,x,.)$. On the other hand, when $H(t,x,.)$ is merely locally Lipschitz continuous, then $D(t,x)=\bb R^n$ if and only if $H(t,x,.)$ is coercive, i.e., $\lim_{|p|\ra +\infty}\frac{H(t,x,p)}{|p|}=+\infty$ (cfr. \cite[Theorem 11.8]{rockafellar2009variational}).

			\item Let $r>0$ such that $x\in  B(0,r)$. From Remark \ref{remark_stima_H_star_ell}, \rif{stima_norma_delta}, and \rif{H_star_min_ell} we get for all $u\in \bb B$
			
			$\ell(t,x,u)-H^*(t,x,f(t,x,u))\mineq 10 m (\tilde c\rt+C_r\rt)(1+|x|)+|H(t,0,0)|+\gamma(t,x).$

			\item It is necessary to point out that, although not in the main interest of this paper, by using the same arguments that those proposed in \cite{frankowskasedrak2014stablerephami} (cfr. \cite{rampazzo2005faithful}), the representation of Theorem \ref{theo_rep_H} is faithful.

		}
		
	\end{remark}
		
	\section{{{Hamilton-Jacobi-Bellman equations on infinite horizon}}}\label{sec_inf_hor_control_prob}

	For any $a\in \bb R$, $b\in \bb R\cup \{+ \infty\}$, with $a<b$, and $A\subset \bb R^k$ we take the following notation:

	\marginpersonal{3}{0}{\hspace{-7mm}$\bullet$ \hspace{4mm}$\scr L^1(a,b;A)$ denotes the normed space of all $A$-valued Lebesgue integrable functions on $[a,b)$ (we write $w\in\scr L^1_{\ttnn{loc}}(a,b;A)$ if $w\in \scr L^1(c,d;A)$ for any $[c, d]\subset [a,b)$)\footnote{If $w\in \scr L^1_{\ttnn{loc}}(a,+\infty;\bb R)$, we denote $\int_{a}^{\infty}w\rs\,ds:=\lim_{b\ra\infty}\int_{a}^{b}w\rs\,ds$, provided this limit exists.};
	}
	\marginpersonal{3}{0}{\hspace{-7mm}$\bullet$ \hspace{3mm}$\scr W^{1,1}_{loc}(a,b;A)$ is the normed space of all $A$-valued locally absolutely continuous functions on $\ttnn{cl }[a,b)$;
	}
	\marginpersonal{3}{0}{\hspace{-7mm}$\bullet$ \hspace{3mm}$\scr L_{\ttnn{loc}}$ is the set of  all $f\in \scr L^1_{\ttnn{loc}}(0,+\infty;\bb R^+)$ such that $\lim_{\sigma\ra 0}\, \sup \{\int_{J} {f(\tau)}\,d\tau \,|$ $\, J\subset  \bb R^+,\, \mu(J)\mineq \sigma\}\,=0$. 
	}
	
	\noindent In this section we consider a closed non-empty subset $\Omega\subset \bb R^n$.	
	
	\subsection{Weak solutions}
	We denote by {H.2.1,3,4} the assumptions in {H.1.1,3,4}, and by {H.2.2} the assumption {H.1.2} with $C_r\ccd\mineq C\ccd$ for any $r>0$ and a suitable $C\in \scr L_{\ttnn{loc}}$. We consider the further two conditions on the Hamiltonian:
	
	\marginpersonal{3.2}{0}{\hspace{-10mm}\ttnn{H.2.5}.  there exist $\tilde \varphi\in \scr L^1_{\ttnn{loc}}(0,+\infty;\bb R^+)$ and $\varphi\in \scr L^1(0,+\infty;\bb R)$ such that $\ttnn{ for a.e. } t\mageq 0$, for all $x\in \bb R^n$, and all $q\in D(t,x)$
		
		\vspace{2mm}
		
		$\varphi\rt	\mineq  {H^*(t,x,q)} \mineq \tilde \varphi\rt(1+|x|);$
		
		\vspace{2mm}
	}

	\marginpersonal{3.2}{0}{\hspace{-10mm}\ttnn{H.2.6}. there exists $ \psi\in \scr L_{\ttnn{loc}}$ such that for a.e. $t\mageq 0$, for all $x\in \ttnn{bdr }\Omega,$ and all $q\in D(t,x)$
		
		\vspace{2mm}
		
		$\modulo{q}+\modulo{H^*(t,x,q)}  \mineq \psi\rt.$
		
		\vspace{2mm}}

	\begin{definition}\label{definition_weak_sol}\rm
		A lower semicontinuous function $V:\bb R^+ \times \Omega\ra \bb R\cup \{\pm \infty\}$ is called a \textit{weak} $-$ or \textit{viscosity} $-$ \textit{solution} of the HJB equation \rif{def_weak_sol}	if there exists a set $C\subset (0,+\infty)$, with $\mu((0,+\infty\backslash C))=0$, such that
		\eee{
			&-p_t+H(t,x,-p_x)\mageq 0\\
			& \forall\,(p_t,p_x)\in \partial_- V(t,x),\,   \forall (t,x)\in \ttnn{dom}\,V\cap (C\times \ttnn{bdr }\Omega),
		}
		and
		\eee{
			&-p_t+H(t,x,-p_x)= 0\\
			&\forall\,(p_t,p_x)\in \partial_- V(t,x)   ,\, \forall (t,x)\in \ttnn{dom}\,V\cap(C\times \ttnn{int}\,\Omega).
		}
		
	\end{definition}
	
	\noindent For all $t\in \bb R^+$, $x\in \bb R^n$, and $u\in \bb R^{n+1}$ we denote by $(f(t,x,u),\ell(t,x,u)):=\phi(t,x,u)$ the representation of the Hamiltonian given by Theorem \ref{theo_rep_H}, and by $\scr U_\Omega(t,x)$ the (possibly empty) set of all trajectory-control pairs $(\xi,u):[t,+\infty)\ra \bb R^n\times  \bb R^{n+1}$ such that $u\ccd$ is measurable and
	\sistemaref{sistemacontrollo}{
		\xi'\rs=f(s,\xi\rs,u\rs) &s\in[t,+\infty)\ttnn{ a.e.}\\
		\xi(t)=x\\
		u\ccd \subset \bb B,\quad \xi\ccd\subset \Omega.
	}
	The value function $v:\bb R^+\times \bb R^n\ra \bb R\cup \{\pm \infty\}$ associated to the representation $(f,\ell)$ is defined by
	
	$
		v(t,x):= \inf \, \graffe{ \;\int_t^{+\infty} \ell(s,\xi\rs,u\rs)\,ds   \;|\; (\xi\ccd,u\ccd)\in \scr U_\Omega(t,x)  },
	$
	
	\noindent where $v(t,x)=+\infty$ if $\scr U_\Omega(t,x)=\emptyset$, by convention. In the following we consider the outward pointing condition (briefly O.P.C.):
	\marginpersonal{4}{1}{\hspace{-12.152mm}\ttnn{O.P.C.}\hspace{0mm} there exist $\eta, \,r,\,M>0$ such that\\
	$\ttnn{for a.e. } t>0, \forall y\in \partial\Omega+ \eta \bb B,\, \forall q\in D(t,y)$, satisfying $ \inf_{n\in N^{\eta}_\Omega(y)}    \ps{n}{q}  \mineq 0,$\\
	there exists $w\in\,D(t,y)\cap B(q,M)$ such that $\inf_{n\in N^{\eta}_\Omega(y) } \{ \ps{n}{w},\,\ps{n}{w-q} \}  \mageq r.$
	}

	\begin{theorem}\label{main_cor_viscosity}
		Assume \ttnn{H.2.1-6} and \ttnn{O.P.C.}

		\noindent Let $V:\bb R^+ \times \Omega\ra \bb R$ be a locally Lipschitz continuous function satisfying the vanishing condition at infinity \rif{vanishing_cond}.

		Then the following statements are equivalent:
		\enurom{
			\item $V=v$;
			
			\item  $V$ is weak solution of the \ttnn{HJB} equation \ttnn{\rif{def_weak_sol}}.

		}

	\end{theorem}
	
	The proof of Theorem \ref{main_cor_viscosity} is given in Section 5.3.

	\begin{remark}\label{remark_dopo_theo_viscosity}\rm
		
		\
		
		\enualp{

			\item We would like to underline that condition O.P.C. is helpful to construct feasible trajectories for infinite horizon control problems. More precisely, it provides uniform neighboring feasible trajectories results (cfr. \cite{bascofrankowska2018lipschitz}), on any compact interval $[0,T]\subset \bb R^+$, for the dynamics $F(s,x)=-f(T-s,x,\bb B)$. Such results basically says that any absolutely continuous trajectory $\xi\ccd$ starting from a point in $\Omega$ and solving the differential inclusion $\xi\ccd\in F(.,\xi\ccd)$ can be approximated by a sequence of trajectories which remain in the interior of the state constraints $\Omega$.

			\item For existence results of the HJB in \rif{def_weak_sol} in the free-constraints case we refer to \cite{bardicapuzzodolcetta1997hjbeviscosity} and the literature therein. However, for infinite horizon optimal control problems under state constraints, O.P.C. does not ensure the existence of feasible trajectories. Existence results under state constraints are investigate in \cite{bascofrank2019hjbinfhor} in the case of Lagrangian with discount factor $e^{-\lambda t}$ and a suitable inward pointing condition.
			
			\item Assumption H.2.6 can be skipped for finite horizon settings in order to recover neighbouring feasible trajectory results (cfr. Remark \ref{remark_dopo_theo_viscosity}-(b) above). However, for infinity horizon control problem, such condition is sufficient to ensure uniform neighbouring feasible trajectory theorems (cfr. \cite{bascofrankowska2018lipschitz}).
			
		}
	\end{remark}

	\subsection{Locally absolutely continuity of epigraph}	
	For all $t\in \bb R^+$, $x\in \bb R^n$, and $u\in \bb R^{n+1}$ we put	$L(t,x,u):=H^*(t,x,f(t,x,u))$.
	\begin{lemma}\label{lemma_ipotesi_su_L}
		Assume \ttnn{H.2.1-5}.
		
		Then:
		\margin{\hspace{-6mm}\ttnn{(i)} for all $ x\in \bb R^n$ the mappings $f(.,x,.)$ and $\ell(.,x,.)$ are Lebesgue-Borel measurable and there exists $\phi\in \scr L^1(0,+\infty;\bb R)$ such that $\ell(t,x,u)\mageq \phi\rt$ for a.e. $t\mageq 0$ and all $x\in \bb R^n$ and $u\in \bb R^{n+1}$;
		}
		\margin{\hspace{-7mm}\ttnn{(ii)} there exists $ c\in \scr L^1_{loc}(0,+\infty;\bb R^+)$  such that $	\modulo{f(t,x,u)}+\modulo{\ell(t,x,u)}\mineq c\rt(1+\modulo{x})$ for a.e. $t\mageq 0$ and for all $x\in \bb R^n$, $u\in \bb B$;
		}
		\margin{\hspace{-8mm}\ttnn{(iii)} for a.e. $t\mageq 0$ and all $x\in \bb R^n$, the set-valued map $\bb R^n \ni y\rightsquigarrow \{(f(t,y,u),$ $\ell(t,y,u))\,|\, u\in \bb B\}$ is continuous with closed images;
		}
		\margin{\hspace{-8mm}\ttnn{(iv)} there exists $ k\in \scr L_{\ttnn{loc}}$ such that $|f(t,x,u)-f(t,y,u)|+\modulo{\ell(t,x,u)-\ell(t,y,u)}\mineq k\rt |x-y|$ for a.e. $t\mageq 0$ and for all $x,\,y\in \bb R^n$, $u\in \bb B$.
		}

	\end{lemma}
	\begin{proof}
		All the conclusions follows from our assumptions and Theorem \ref{theo_rep_H}.
	\qed\end{proof}

	Next, we recall the definition of locally absolutely continuity for set-valued maps.
	
	\begin{definition}\rm
		A set-valued map $S:\bb R^+\rightsquigarrow \bb R^k$  is called \textit{locally absolutely continuous} (briefly \textit{l.a.c.}) if it takes non-empty closed images and every $\eps>0$, any $[t,T]\subset \bb R^+$, and any compact subset $K\subset \bb R^k$, there exists $\delta>0$ such that for any finite partition $t\mineq t_1<\tau_1\mineq t_2<\tau_2\mineq ...\mineq t_m<\tau_m\mineq T$ of $[t,T]$ satisfying $	\sum_{i=1}^m
		(\tau_i-t_i)<\delta$ holds
		
		$\sum_{i=1}^m  \max\graffe{ e({S(t_i)},S(\tau_i)\cap K)\,,\, e({S(\tau_i)},S(t_i)\cap K)}<\eps$,

		\noindent where $e(E,E'):=\inf \graffe{r>0\,:\, E'\subset E+r \bb B}$ for all $E, E'\subset \bb R^k$ ($\inf\,\emptyset :=+\infty, \ttnn{ by convention}$).
	\end{definition}
	
	\begin{proposition}\label{v_uguale_inf_int_H_star}
		Assume \ttnn{H.2.1-6} and \ttnn{{O.P.C.}} Denote by $W:\bb R^+\times \bb R^n\ra \bb R\cup \graffe{+\infty}$ the value function of the following infinite horizon control problem under state constraints:	\ttnn{minimize  $\, \int_t^{+\infty} L(s,\xi\rs,u\rs)\,ds  $ over all $(\xi\ccd,u\ccd)\in \scr U_\Omega(t,x) $}	\ttnn{such that $\xi(t)=x$, where $(t,x)\in \bb R^+\times \bb R^n$ is the initial datum.}

		Then:

		\margin{\hspace{-6mm}\ttnn{(i)} $W(t,x)=\inf \;\{\,\int_t^{+\infty} H^*(s,\xi\rs,\xi'\rs)\,ds  \,|\, \xi\in \scr  W^{1,1}_{loc}(t,+\infty;\bb R^n),\,\xi(t)=x,$ $\, \xi\ccd\subset \Omega \}$, for any $(t,x)\in \bb R^+\times \bb R^n$ such that $\scr U_\Omega(t,x)\neq \emptyset$;
		}
		
		\margin{\hspace{-7mm}\ttnn{(ii)} $W$ and $v$ are lower semicontinuous and $t\rightsquigarrow \ttnn{epi }W(t,.)$ is l.a.c.;
		}
		
		\margin{\hspace{-8mm}\ttnn{(iii)} there exists a set $ C'\subset (0,+\infty)$, with $\mu((0,+\infty)\backslash C'))=0$, such that for any $(t,x)\in \ttnn{dom}\,v \cap (C'\times \ttnn{int}\,\Omega)$
			\eee{
				\forall u\in \bb B,\quad D_\uparrow v(t,x)(-1,-f(t,x, u))\mineq \ell(t,x, u);
			}
		}
		
		\margin{\hspace{-8mm}\ttnn{(iv)} there exists a set $ C''\subset (0,+\infty)$, with $\mu((0,+\infty)\backslash C''))=0$, such that for any $(t,x)\in \ttnn{dom}\,v \cap (C''\times \ttnn{int}\,\Omega)$
			\eee{
				\forall u\in \bb B,\quad -\ell(t,x, u)\mineq D_\downarrow v(t,x)(1,f(t,x, u)).
			}
		}

	\end{proposition}
	\begin{proof}
		The statement (i) is a known fact (cfr. \cite{bardicapuzzodolcetta1997hjbeviscosity}), and for the lower semicontinuity of $W$ and $v$ and the locally absolutely continuity of the epigraph of the value function $W$, under the viability condition \ttnn{{O.P.C.}}, we refer to \cite{bascofrankowska2018lipschitz,bascofrank2019hjbinfhor}. So, (ii) holds.
		
		Let us define for all $ t\in \bb R^+$ and all $x\in \bb R^n$
		\equazioneref{def_G}{
			G(t,x)&:= \{(f(t,x,u),-\ell(t,x,u)-r)\,|\, u\in \bb B,\,\\
			&\hspace{4.95cm} r\in [0,  c\rt(1+\modulo{x})-\ell(t,x,u)]    \}.
		}
		\noindent Next, we prove (iii). Let $j\in \bb N^+$. Recalling Lemma \ref{lemma_ipotesi_su_L}, we apply \cite[Theorem 2.9]{frankowskaplaskrze1995measviabth} to the set-valued map $[1/j,j]\times \bb R^n\times \bb R\ni (s,\xi,\beta)\rightsquigarrow -G(j-s,\xi)\in \bb R^n\times \bb R$  and the measurable selection theorem: there exists a subset $C_j'\subset [1/j,j]$, with $\mu(C_j')=0$, such that for any $(t_0,x_0)\in ((1/j,j]\backslash C_j')\times \ttnn{int}\,\Omega$ and any $u_0\in \bb B$ we can find $t_1\in [1/j,t_0)$ and a trajectory-control pair $((\xi,\beta),(u,r))\ccd$ satisfying
		\sistemanoref{
			(\xi ,\beta )'\rt=(f(t,\xi\rt,u\rt),-\ell(t,\xi\rt,u\rt)-r\rt) &t\in [t_1,t_0]\ttnn{ a.e.}\\
			(u,r)\rt\in \bb B\times [0,c\rt(1+\modulo{\xi\rt})-\ell(t,\xi\rt,u\rt)] &t\in [t_1,t_0]\ttnn{ a.e.}\\
			(\xi,\beta)(t_0)=(x_0,0)\\
			(\xi,\beta)'(t_0)=(f(t_0,x_0,u_0),-\ell(t_0,x_0,u_0)),
		}
		and	$\xi([t_1,t_0])\subset \Omega.$ Hence, if $(t_0,x_0)\in \ttnn{dom}\,v$, by the dynamic programming principle it follows that $\dfrac{v(s,\xi\rs)-v(t_0,x_0)}{t_0-s}\mineq \dfrac{1}{t_0-s}(\beta(s)-\beta(t_0))$ for all $s\in [t_1,t_0]$. Passing to the lower limit as $s\ra t_0-$ and using the lower semicontinuity of $v$, we conclude $D_\uparrow v(t_0,x_0)(-1,-f(t_0,x_0,u_0))\mineq \ell(t_0,x_0,u_0).$
		
		\noindent Since $u_0\in \bb B$ is arbitrary, the statement (iii) follows with $C'=(0,+\infty)\backslash \cup_{j\in \bb N}C'_j$. 	The statement (iv) holds as well  arguing in a similar way.
	\qed\end{proof}
	\begin{remark}
	Notice that, from \cite[Proposition 4.4]{bascofrank2019hjbinfhor} and Proposition \ref{v_uguale_inf_int_H_star}, under the assumptions H.2.1-6 and O.P.C.,  the set-valued map $t \rightsquigarrow \ttnn{epi}\, v(t,.)$ is l.a.c. even though $v$ may be discontinuous.

	\end{remark}

	\noindent {\bfseries{\ref{sec_inf_hor_control_prob}.2.}} \textbf{Viability of hypograph.} Next lemma provides a viability result of the hypograph of weak solutions.

	\begin{lemma}\label{lemma_parte_interna_2}
		Assume \ttnn{H.2.1-5}. Let $V:\bb R^+ \times \Omega\ra \bb R \cup \graffe{+\infty}$ be such that
		\eee{
			t\rightsquigarrow \graffe{(x,\lambda )\in \Omega\times \bb R\,|\, \lambda\mineq V(t,x)<+\infty }\ttnn{ is l.a.c.}
		}
		If there exists a set $E'\subset (0,+\infty)$, with $\mu((0,+\infty)\backslash E')=0$, such that
		\equazioneref{H_parte_interna}{
			&-p_t+\sup_{u\in \bb B} \graffe{\ps{f(t,x,u)}{-p_x} +q\ell(t,x,u)  }\mineq 0\\
			&  \forall\,(p_t,p_x,q)\in {T_{\ttnn{hypo}\, V}(t,x,V(t,x))}^+,\forall (t,x)\in \ttnn{dom}\,V\cap(E'\times \ttnn{int}\,\Omega),
		}
		then for all $0< \tau_0<\tau_1$ and any feasible trajectory-control pair $( \xi\ccd, u\ccd)$ on $I=[\tau_0,\tau_1]$, with $ \xi([\tau_0,\tau_1])\subset \ttnn{int}\,\Omega$ and $(\tau_0,\xi(\tau_0))\in \ttnn{dom}\,V$, we have
		\eee{
			( \xi\rt,V(\tau_0,\xi(\tau_0))-\int_{\tau_0}^{t} \ell(s, \xi\rs, u\rs)ds)\in \ttnn{hypo}\,V(t,.)\quad \forall t\in [\tau_0,\tau_1].
		}

	\end{lemma}
	
	\begin{proof}
		
		Notice that, by the separation theorem, \rif{H_parte_interna} is equivalent to $\graffe{1}\times G(t,x)\subset {\overline{\ttnn{co}}}\,\,T_{\ttnn{hypo}\,V}(t,x,\beta)$ for all $\beta\mineq  V(t,x)$ and all $(t,x)\in (E'\times \ttnn{int}\,\Omega)\cap \ttnn{dom}\,V$, where we defined

		Let $0< \tau_0<\tau_1$ and put $Q(s):=\ttnn{hypo}\,V( s,.)$ for any $s\in [\tau_0,\tau_1]$. We have
		\equazioneref{1_f_L}{
			(1,  f(s,x,u),  -\ell(s,x,u))\in  {\overline{\ttnn{co}}}\,T_{\ttnn{graph}\,Q}(s,x,\beta)
		}
		for a.e. $s\in[\tau_0,\tau_1]$, any $(x,\beta)\in Q(s)\cap (\ttnn{int}\,\Omega\times \bb R)$, and any $u\in \bb B$.  Consider a trajectory-control pair $( \xi\ccd, u\ccd)$ solving \rif{sistemacontrollo} on $I=[\tau_0,\tau_1]$, with $ \xi([\tau_0,\tau_1])\subset \ttnn{int}\,\Omega$ and $(\tau_0,\xi(\tau_0))\in \ttnn{dom}\,V$. We claim that $d_{Q\rs}((\xi\rs, u\rs))=0$ for all $s\in I$, where   $ w(.)$ is the unique solutions of  
		
		$
			w'\rt=-\ell(t, \xi\rt, u\rt)\ttnn{ for a.e. }t\in [\tau_0,\tau_1],\quad w(\tau_0)=V(\tau_0, \xi(\tau_0)).
		$
		
		\noindent Putting $g\rs=d_{Q\rs}((\xi\rs, w\rs))$, applying \cite[Lemma 4.8]{frankowskaplaskrze1995measviabth} and Lemma \ref{lemma_ipotesi_su_L} to the single-valued map $s\rightsquigarrow \graffe{(  f(s,\xi\rs,  u\rs),  -\ell(s,\xi\rs,  u\rs))}$, we have that $g\ccd$ is absolutely continuous. Let for any $s\in I$ the pair $(p\rs,r\rs)\in Q\rs$ be such that
		
		$g\rs=\modulo{(\xi\rs, w\rs)-(p\rs,r\rs)}$.
		
		\noindent We claim that $g\rs= 0$ for all $s\in (\tau_0,\tau_1]$. Suppose, by contradiction, that we can find $T\in (\tau_0,\tau_1]$ with $g(T)>0$. Denoting $t^*=\sup\graffe{t\in [\tau_0,T]\,:\, g\rt=0}$, let $\eps>0$ be such that $p\rs\in \ttnn{int}\,\Omega$ and $g\rs>0$ for any $s\in (t^*,t^*+\eps]$. Consider $s\in (t^*,t^*+\eps)$ where $g\ccd$, $\xi\ccd$, and $  w\ccd$ are differentiable, with $\xi'\rs=  f(s,\xi\rs,  u\rs)$ and $  w'\rs=\ell (s,\xi\rs,  u\rs)$. Consider $(\theta,\lambda)\in T_{\ttnn{graph}\,Q}(s,p\rs,r\rs)$ and $\theta_i\ra \theta,\,\lambda_i\ra \lambda,\,h_i\ra0+$ satisfy $(p\rs,r\rs)+h_i\lambda_i\in Q(s+h_i\theta_i)\quad \ttnn{for all } i\in \bb N.$
		Then, setting $q=(\xi\rs,  w\rs)$ and $\bar q=(p\rs,r\rs)$, we get
		
		$
			g(s+h_i\theta_i)-g(s)	\mineq \modulo{\tonde{\xi(s+h_i\theta_i),  w(s+h_i\theta_i)}-\bar q-h_i\lambda_i}-\modulo{q-\bar q}.
		$
		
		\noindent Dividing this inequality by $h_i$, passing to the limit as $i\ra +\infty$, and putting $\zeta:=\dfrac{q-\bar q}{\modulo{q-\bar q}}$, we have
		\equazioneref{g_primo_min_eq}{
			g'\rs \theta\mineq \ps{\zeta}{\tonde{  f(s,\xi(s),  u\rs), -\ell (s,\xi(s),  u\rs)}\theta-\lambda}.
		}
		Since \rif{g_primo_min_eq} holds for any $(\theta,\lambda)\in T_{\ttnn{graph}\,Q}(s,p\rs,r\rs)$, taking convex combinations of elements in $T_{\ttnn{graph}\,Q}(s,p\rs,r\rs)$ we conclude that \rif{g_primo_min_eq} holds for all $(\theta,\lambda)\in {\overline{\ttnn{co}}}\,T_{\ttnn{graph}\,Q}(s,p\rs,r\rs)$. By \rif{1_f_L}, the inequality \rif{g_primo_min_eq} holds true for $\theta=1$ and $\lambda=\tonde{  f(s,p\rs,  u\rs),-\ell (s,p\rs,  u\rs)}.$
		Therefore, from Lemma \ref{lemma_ipotesi_su_L}-(iv),
		
		$
			g'\rs\mineq k\rs \modulo{\xi\rs-p\rs}\mineq k\rs g\rs.
		$
		
		\noindent From the Gronwall lemma we conclude that $g\ccd= 0$ on $[t^*,t^*+\eps]$, and a contradiction follows.
	\qed\end{proof}

	\noindent {\bfseries{\ref{sec_inf_hor_control_prob}.3.}} \textbf{Proof of Theorem \ref{main_cor_viscosity}.} In this section we provide a proof of Theorem \ref{main_cor_viscosity}.
	
	\begin{proposition}\label{proposition1lipschitz}
		Assume \ttnn{H.2.1-6} and \ttnn{{O.P.C.}}
		Let $V:\bb R^+ \times \Omega\ra \bb R \cup \graffe{+\infty}$ be a lower semicontinuous function, satisfying the vanishing condition at infinity \rif{vanishing_cond}, such that $\ttnn{dom}\,v(t,.)\subset \ttnn{dom}\,V(t,.)\neq \emptyset$ for all large $t>0$ and 
		\equazioneref{hypo_W_abs_cont}{
			t\rightsquigarrow \graffe{(x,\lambda )\in \Omega\times \bb R\,|\, \lambda\mineq V(t,x)<+\infty }\; \textit{is l.a.c.}
		}

		Then the following statements are equivalent:
		\enurom{
			\item $V=v$;
			
			\item  $t\rightsquigarrow \ttnn{epi}\,V(t,.)$ is l.a.c. and there exists $E\subset (0,+\infty)$, with $\mu((0,+\infty)\backslash E)=0$, such that:
			
			\ttnn{(ii.a)} $-p_t+ \sup_{u\in \bb B} \graffe{\ps{f(t,x,u)}{-p_x} +q\ell(t,x,u)  } \mageq 0$
			
			$\qquad\; \forall\,(p_t,p_x,q)\in {T_{\ttnn{epi}\, V}(t,x,V(t,x))}^-,\;\forall (t,x)\in \ttnn{dom}\,V\cap(E\times \Omega),$

			\ttnn{(ii.b)} $-p_t+\sup_{u\in \bb B} \graffe{\ps{f(t,x,u)}{-p_x} +q\ell(t,x,u)  } \mineq 0$
			
			$\qquad\;\forall\,(p_t,p_x,q)\in {T_{\ttnn{hypo}\, V}(t,x,V(t,x))}^+,\;\forall (t,x)\in \ttnn{dom}\,V\cap(E\times \ttnn{int}\,\Omega)$.
		}		
	\end{proposition}
	
	\begin{proof}
		Notice that, by the definition of locally absolutely continuous set-valued map, the hypograph of $V(t,.)$ restricted to $\ttnn{dom}\,V(t,.)$ is closed. To show the equivalence between statements (i) and (ii), we use the following claim: for any $(t,x)\in \bb R^+\times \bb R^n$ with $\scr U_\Omega(t,x)\neq \emptyset$, $v(t,x)$ is equal to the following infimum
		\vspace{-2mm}
		\begin{align*}
		\ttnn{(CV)}\quad 
		\begin{cases}
		\ttnn{$\inf \; \int_t^{+\infty} H^*(s,\xi\rs,\xi'\rs)\,ds  $ over all $\xi\in \scr  W^{1,1}_{loc}(t,+\infty;\bb R^n)$ }\\
		\ttnn{such that $\xi(t)=x$ and $\, \xi\ccd\subset \Omega.$}
		\end{cases}
		\end{align*}		
		Indeed, let $(t,x)\in \bb R^+\times \bb R^n$ such that $\scr U_\Omega(t,x)\neq \emptyset$ and denote by $\alpha(t,x)\in \bb R\cup \graffe{\pm\infty}$ the infimum in (CV) above. From assumption H.2.5 we have that $\alpha(t,x)\neq -\infty.$ If $\alpha(t,x)=+\infty$ then $\alpha(t,x)\mageq v(t,x)$. Assume $\alpha(t,x)\in \bb R$. Fix $\eps>0$ and consider $\xi\in \scr  W^{1,1}_{loc}(t,+\infty;\bb R^n)$ with $\xi(t)=x$ and $\, \xi\ccd\subset \Omega$ satisfying $\int_t^{+\infty} H^*(s,\xi\rs,\xi'\rs)\,ds <\alpha(t,x)+\eps $. We have that $(\xi'\rs,u'\rs)\in \ttnn{graph }H^*(s,\xi\rs,.)$ for a.e. $s\mageq t$, where we put $u\rs:=\int_t^{s} H^*(\tau,\xi(\tau),\xi'(\tau))\,d \tau$ for all $s\mageq t$. Applying now the representation Theorem \ref{theo_rep_H}-(v)$'$ and the measurable selection theorem, we have that there exists a measurable function $w:[t,+\infty)\ra \bb B$ such that $(\xi'\rs,u'\rs)=(f(s,\xi\rs,w\rs),\ell(s,\xi\rs,w\rs))$ for a.e. $s\mageq t$. We get	
		\eee{
			\int_t^{+\infty} H^*(s,\xi\rs,\xi'\rs)\,ds=\int_t^{+\infty} u'\rs\,ds =\int_t^{+\infty} \ell(s,\xi\rs,w\rs)\,ds\mageq v(t,x).
		}
		So, $\alpha(t,x)+\eps > v(t,x)$. Since $\eps$ is arbitrary, we deduce that $\alpha(t,x)\mageq v(t,x)$. Arguing in analogous way as above and using \rif{H_star_min_ell}, we get $\alpha(t,x)\mineq v(t,x)$.
		
		Next, we show (i)$\Longleftrightarrow$(ii).	Assume $(i)$. Applying Lemma \ref{v_uguale_inf_int_H_star} and the claim, we have that $t\rightsquigarrow \ttnn{epi }v(t,.)$ is l.a.c. for any $x\in \Omega$. Now, from Proposition \ref{v_uguale_inf_int_H_star}-(iv) and the claim, we can find a subset $C\subset (0,+\infty)$, with $\mu(C)=0$, such that for any $(t,x)\in ((0,+\infty)\backslash C)\times \ttnn{int}\,\Omega$ we have $-\ell(t,x,u)\mineq D_\downarrow v(t,x)(1,f(t,x,u))$ for all $u\in \bb B$. Hence, from \cite[Proposition 6.1.4]{aubin2009set},
		we get
		
		$
			(1,f(t,x,u),-\ell(t,x,u))\in T_{\ttnn{hypo}\,v}(t,x,v(t,x))
		$
		
		\noindent for any $u\in \bb B$. Then we get
		
		$
			-p_t+\sup_{u\in \bb B}{\ps{f(t,x,u)}{-p_x} +q\ell(t,x,u)  }\mineq 0
		$
		
		$
			 \forall (p_t,p_x,q)\in {T_{\ttnn{hypo}\, v}(t,x,v(t,x))}^+.
		$
		
		\noindent Hence, statement (ii.b) holds. Using a similar argument and applying Lemma \ref{lemma_ipotesi_su_L} and \cite[Theorem 3.3]{bascofrank2019hjbinfhor}, we get (ii.a). Thus, (ii) follows.
		
		Now, assume $\ttnn{(ii)}$. From condition \rif{vanishing_cond} and \cite[Theorem 3.3]{bascofrank2019hjbinfhor} and its proof, it is just sufficient to show the following: there exists $ C\subset (0,+\infty)$, with $\mu((0,+\infty)\backslash C)=0$, such that
		\equazioneref{ii-a'}{
			&\forall (t,x)\in \ttnn{dom}\,V\cap(C \times \ttnn{int }\Omega),\, \forall \, u\in \bb B,\\
			& D_\uparrow V(t,x)(-1,-f(t,x, u))\mineq \ell(t,x, u).
		} 
		Recalling the definition of $G(.,.)$ given in \rif{def_G}, applying Lemma \ref{lemma_ipotesi_su_L} and  \cite[Theorem 2.9]{frankowskaplaskrze1995measviabth} to the set-valued maps $[0,j]\times \bb R^n\times \bb R\ni (s,\xi,\beta)\rightsquigarrow -G(j-s,\xi)\in \bb R^n\times \bb R$ where $j\in \bb N$, and from the measurable selection theorem, we can find a family of subsets $C_j'\subset(0,j)$, with $\mu( C_j')=0$ for all $j\in \bb N$, such that for any $(t_0,x_0)\in ((0,+\infty)\backslash \cup_{j\in \bb N}C_j')\times \ttnn{int}\,\Omega$ and any $u_0\in \bb B$, there exists $t_1\in (0,t_0)$ and a trajectory-control pair $((\xi,\beta),(u,r))\ccd$ satisfying
		\sistemanoref{
			(\xi ,\beta )'\rt=(f(t,\xi\rt,u\rt),-\ell(t,\xi\rt,u\rt)-r\rt) &t\in [t_1,t_0]\ttnn{ a.e.}\\
			(u,r)\rt\in \bb B\times [0,c\rt(1+\modulo{\xi\rt})-\ell(t,\xi\rt,u\rt)] &t\in [t_1,t_0]\ttnn{ a.e.}\\
			\xi([t_1,t_0])\subset \ttnn{int}\,\Omega,
		}
		with initial condition and final velocity
		
		$
			(\xi,\beta)(t_0)=(x_0,0),\quad 
			(\xi,\beta)'(t_0)=(f(t_0,x_0,u_0),-\ell(t_0,x_0,u_0)).
		$
			
		\noindent Hence, applying Lemma \ref{lemma_parte_interna_2} and taking a sequence $s_i\in (t_1,t_0)$ with $s_i\ra t_0-$, we get $V(s_i,\xi(s_i))-\int_{s_i}^{t_0}\ell(s,\xi\rs,u\rs)\,ds\mineq V(t_0,x(t_0))$ for all $ i\in \bb N$.
		So,
		\eee{
			V(s_i,\xi(s_i))-V(t_0,x_0)\mineq \int_{s_i}^{t_0}\ell(s,\xi\rs,u\rs)\,ds \mineq \beta(s_i)\quad \forall i\in \bb N.
		}
		Dividing by $t_0-s_i$ and passing to the lower limit as $i\ra\infty$, we get \rif{ii-a'} with $C=(0,+\infty)\backslash \cup_{j\in \bb N}C_j'$, and the proof is complete.
	\qed\end{proof}

	\noindent {\it Proof of Theorem \ref{main_cor_viscosity}}
		Let $V:\bb R^+ \times \Omega\ra \bb R$ be a locally Lipschitz continuous function and $(t,x)\in \bb R^+\times \Omega$. Notice that, from the locally Lipschitz continuity of $V$, the following set-valued maps $	t\rightsquigarrow \ttnn{epi}\,V(t,.)$ and $  	t\rightsquigarrow \ttnn{hypo}\,V(t,.)$ are locally absolutely continuous and, since $\partial_-V (t,x)$, $\partial_+V (t,x)$ are non-empty closed sets, it is straightforward to see that $\cup_{\lambda \mageq 0}\,\lambda (\partial_-V (t,x),-1)$, $\cup_{\lambda \mageq 0}\,\lambda (\partial_+V (t,x),-1)$ are closed too.  Now, we claim that:
		\equazioneref{sub_diff_and_polar_cones}{
			\cup_{\lambda \mageq 0}\,\lambda (\partial_-V (t,x),-1)= T_{\ttnn{epi}\, V}(t,x,V(t,x))^-.
		}
		Indeed, from the following well known relation (cfr. \cite[Chapter 6.4]{aubin2009set})
		\equazioneref{con_cone_epi_2}{
			&\zeta\in \partial_-V (t,x)\Longleftrightarrow    (\zeta,-1)\in T_{\ttnn{epi}\, V }(t,x,V(t,x))^-,
		}
		it follows that $	\cup_{\lambda \mageq 0}\,\lambda (\partial_-V (t,x),-1)\subset T_{\ttnn{epi}\, V}(t,x,V(t,x))^-$. On the other hand, let $(\zeta,q)\in T_{\ttnn{epi}\, V}(t,x,V(t,x))^-$. Since $(0,\delta)\in T_{\ttnn{epi}\, V}(t,x,V(t,x))^-$ for all $\delta\mageq 0$,   we have $q\mineq 0$. If $q<0$, $(\zeta/|q|,-1)\in  T_{\ttnn{epi}\, V}(t,x,V(t,x))^-$ and, applying \rif{con_cone_epi_2}, $\frac{\zeta}{|q|}\in \partial_-V (t,x)$. So, $(\zeta,q)\in \cup_{\lambda \mageq 0}\,\lambda (\partial_-V (t,x),-1))$. If $q=0$, consider $\bar \zeta\in  \partial_-V (t,x)$. Then $(\bar \zeta,-1)\in  T_{\ttnn{epi}\, V}(t,x,V(t,x))^-$, and, from the convexity of the polar cone, $(r\bar \zeta+(1-r)\zeta,-r)\in T_{\ttnn{epi}\, V}(t,x,V(t,x))^-$ for all $0<r<1$. Arguing as above, we conclude that $(r\bar \zeta+(1-r)\zeta,-r)\in \cup_{\lambda \mageq 0}\,\lambda (\partial_+V (t,x),-1)$, and the claim \rif{sub_diff_and_polar_cones} follows. Using the same argument as above, we have also
		\equazioneref{sub_diff_and_polar_cones_2}{
			\cup_{\lambda \mageq 0}\,\lambda (\partial_+V (t,x),-1)= T_{\ttnn{hypo}\, V}(t,x,V(t,x))^+.
		}

		Finally, from \rif{sub_diff_and_polar_cones}, \rif{sub_diff_and_polar_cones_2}, Proposition \ref{proposition1lipschitz}, and the representation Theorem \ref{theo_rep_H}, the conclusion follows.
	\qed

	\bibliographystyle{plain}
	\bibliography{BIBLIO_VB_rep_hamil}

\end{document}